\documentclass[12pt]{amsart}
\usepackage{amssymb, enumerate, mdwlist, amsthm, amsmath, bm}
\usepackage[dvipdfmx]{xcolor,hyperref}
\hypersetup{
    colorlinks=true,
    citecolor=blue,
    linkcolor=red,
    urlcolor=orange,
}

\newtheorem{thm}{Theorem}
\newtheorem{prop}{Proposition}
\newtheorem{conj}{Conjecture}
\newtheorem{prob}[conj]{Problem}
\newtheorem{lemma}{Lemma}
\newtheorem{claim}{Claim}
\newtheorem{fact}{Fact}
\newtheorem{preex}{Example}
\newenvironment{example}%
  {\begin{preex}\upshape}{\end{preex}}
\newtheorem{predef}{Definition}
\newenvironment{defn}%
  {\begin{predef}\upshape}{\end{predef}}
\newtheorem{presu}{Setup}
\newenvironment{setup}%
  {\begin{presu}\upshape}{\end{presu}}

\def\E{\mathbb E}

\def\R{\mathbb R}
\def\AA{\mathcal A}
\def\BB{\mathcal B}
\def\SS{\mathcal S}
\def\GG{\mathcal G}
\def\la{{\langle}}
\def\ra{{\rangle}}
\def\one{{\bf 1}}
\def\zero{{\bf 0}}
\def\vv{{\bm v}}
\def\pp{{\bm p}}

\def\FF{{\mathcal F}}

\def\phi{\varphi}
\def\hat{\widehat}
\title[Application of hypergraph Hoffman's bound]{Application of hypergraph Hoffman's bound\\ to intersecting families}
\author{Norihide Tokushige}
\date{\today}
\begin{document}
\maketitle
\begin{abstract}
Using the Filmus--Golubev--Lifshitz method \cite{FGL} to bound the independence
number of a hypergraph, we solve some problems concerning multiply intersecting
families with biased measure. Among other results we obtain a stability result
of a measure version of the Erd\H os--Ko--Rado theorem for multiply intersecting
families.
\end{abstract}
\section{Introduction}
It is an important problem to estimate the size of a maximum independent set
in a graph, and Hoffman's bound\footnote{
See \cite{Haemers} for the history of the bound, and some related results
including Delsarte's LP bound and Lov\'asz' theta bound.
} is one the most useful algebraic tools for the problem.
Recently, Filmus, Golubev, and Lifshitz \cite{FGL} extended the bound to 
hypergraphs. In this paper we apply these bounds to some problems 
concerning multiply intersecting families with biased measures.

We start with the easiest and the most basic result about intersecting
families with a biased measure. Let $V$ be a finite set and let 
$\AA\subset 2^V$ be a family of subsets of $V$. 
For a fixed real number $p$ with $0<p<1$ we define the $p$-biased measure 
of the family $\AA$ by
\[
 \mu_p(\AA) = \sum_{A\in\AA} p^{|A|} (1-p)^{|V|-|A|}.
\]
By definition it follows that $\mu_p(2^V)=1$.
We say that $\AA$ is \emph{intersecting} if $A\cap A'\neq\emptyset$ 
for all $A,A'\in\AA$.
A typical intersecting family is 
\[
\SS=\{A\in 2^V:v\in A\}
\]
for some fixed $v\in V$. This family is called a \emph{star} centered at 
$v$. The star can be rewritten as $\{\{v\}\cup B:B\in 2^{W}\}$ where
$W=V\setminus\{v\}$, and it follows that
\[
 \mu_p(\SS)=\sum_{A\in\SS} p\cdot p^{|A|-1} (1-p)^{|V|-|A|}
=p \sum_{B\in 2^{W}} p^{|B|} (1-p)^{|W|-|B|} = p.
\]
Indeed, it is not difficult to show that if $p\leq \frac12$ and
$\AA\subset 2^V$ is intersecting, then $\mu_p(\AA)\leq p$,
see e.g., \cite{AK,FFFLO}, or Chapter~12 in \cite{FT2018}.
We will extend this result in several ways.

To state our problems and results we need some more notation and 
definitions. Let $n, r$ be positive integers and let 
$[n]:=\{1,2,\ldots,n\}$. We say that 
a family of subsets $\AA\subset 2^{[n]}$ is \emph{$r$-wise intersecting} if
$A_1\cap A_2\cap\cdots\cap A_r\neq\emptyset$ for all $A_1,A_2,\ldots,A_r\in\AA$.
Let $1>p_1\geq p_2\geq\cdots\geq p_n>0$ be real numbers, and let 
$\pp=(p_1,p_2,\ldots,p_n)$. We define the $\pp$-biased measure 
(or $\mu_{\pp}$-measure) $\mu_{\pp}:2^{[n]}\to(0,1)$ by
\[
 \mu_{\pp}(A):=\prod_{i\in A}p_i\prod_{j\in[n]\setminus A}(1-p_j)
\]
for $A\in 2^{[n]}$, and for a family $\AA\subset 2^{[n]}$ we define
\[
 \mu_{\pp}(\AA) := \sum_{A\in\AA} \mu_{\pp}(A).
\]
The star centered at $i\in[n]$ is an $r$-wise intersecting family with
$\mu_{\pp}(\AA)=p_i$.

Fishburn et al.\ \cite{FFFLO} 
studied the maximal $\mu_{\pp}$-measure for 2-wise intersecting
families using combinatorial tools. Then Suda et al.\ \cite{STT} extended 
their result to cross-intersecting families (see Theorem~\ref{STT-thm} in the
last section) by solving a semidefinite programming problem, 
and posed the following conjecture.

\begin{conj}[\cite{STT}]\label{conj1}
Let $1>p_1\geq p_2\geq\cdots\geq p_n>0$ and $\pp=(p_1,p_2,\ldots,p_n)$. 
Let $p_3<\frac12$. If $\AA\subset 2^{[n]}$ is a $2$-wise intersecting family,
then $\mu_{\pp}(\AA)\leq p_1$. Moreover, if $p_1>p_3$, or $p_1<\frac12$, then
equality holds if and only if $\AA$ is a star centered at some
$i\in[n]$ with $p_1=p_i$.
\end{conj}

This conjecture is true if the condition $p_3\leq\frac12$ is replaced with
$p_2\leq\frac12$, which is proved in \cite{FFFLO} and \cite{STT}. 
In this paper we apply
Hoffman's bound to show the following result which supports the conjecture.

\begin{thm}\label{thm1}
Let $1>p_1\geq p_2\geq\cdots\geq p_n>0$ and $\pp=(p_1,p_2,\ldots,p_n)$. 
Let $p_3<\frac12$. Suppose that $p_1\leq\frac12$ or $1-p_2>p_3$.
If $\AA\subset 2^{[n]}$ is a $2$-wise intersecting family,
then $\mu_{\pp}(\AA)\leq p_1$. 
Moreover equality holds if and only if 
$\AA$ is a star centered at some $i\in[n]$ with $p_1=p_i$.
\end{thm}

Frankl and the author \cite{FT2003} studied $r$-wise intersecting families 
with a $\mu_{\pp}$-measure where $\pp=(p,p,\ldots,p)$, and proved that if 
$p<\frac{r-1}r$ then $\mu_{\pp}(\AA)\leq p$ for every $r$-wise intersecting
family $\AA\subset 2^{[n]}$. The proof was combinatorial. Filmus et al.\ 
\cite{FGL} gave
a new proof by extending Hoffman's bound to $r$-uniform hypergraphs 
($r$-graphs).
In this paper we further extend their method to obtain the following result.

\begin{thm}\label{thm2}
Let $1>p_1\geq p_2\geq\cdots\geq p_n>0$ and $\pp=(p_1,p_2,\ldots,p_n)$. 
If $\frac23>p_2$ and $\AA\subset 2^{[n]}$ is a $3$-wise intersecting family,
then $\mu_\pp(\AA) \leq p_1$. 
Moreover equality holds if and only if 
$\AA$ is a star centered at some $i\in[n]$ with $p_1=p_i$.
\end{thm}

Friedgut \cite{Friedgut} studied the case $r=2$ and $\pp=(p,p,\ldots,p)$, and 
found a stability result. We combine his method with FGL bound to get a 
stability result for the case $r=3$.

\begin{thm}\label{thm3}
Let $0<p<\frac 23$ be fixed, and let $\pp=(p,p,\ldots,p)$.  
Then there exists a positive constant $\epsilon_p$ such that the following 
holds for all $0<\epsilon<\epsilon_p$.
If $\AA\subset 2^{[n]}$ is a $3$-wise intersecting family with 
$\mu_{\pp}(\AA)=p-\epsilon$, then there exists a star $\BB\subset 2^{[n]}$ such
that 
\begin{itemize}
\item[(i)] if $p\leq\frac12$ then $\AA\subset\BB$, and
\item[(ii)] if $p>\frac12$ then 
$\mu_{\pp}(\AA\triangle\BB)<(C_p+o(1))\epsilon$, where
$C_p=\frac{16p(1-p)^2}{(2p-1)(3-4p)}$, and the $o(1)$ term vanishes as
$\epsilon\to 0$.
\end{itemize}
\end{thm}

Finally we mention that there are different and more combinatorial approaches 
to the related problems concerning weighted intersecting families, 
see, e.g., \cite{BE} and \cite{B}.

In Section~2 we prepare tools for the proofs, and then we prove
Theorems~\ref{thm1}--\ref{thm3} in Section~3. In Section~4 we discuss easy
generalization and some related problems.

\section{Preliminaries}\label{sec:prelim}
In this section we collect some tools used to prove our results. In the
first two subsections we reproduce the proof of Hoffman's bound for a 
hypergraph established by Filmus, Golubev, and Lifshitz, for in the next section
we will use not only the bound itself but also some equalities and inequalities
appeared in the proof. 
Our formulation and definitions follow those in \cite{Filmus}, which are 
slightly different from \cite{FGL} (but of course essentially the same).
In the last subsection we present some basic facts
about spectral information related to families with the $\mu_\pp$-measure.

Remark for notation: In this paper we often identify a set and 
its indicator. Let $V$ be a finite set and 
let $\{0,1\}^V$ denote the set of boolean functions from $V$ to $\{0,1\}$.
For $I\in 2^V$ we write $\one_I\in\{0,1\}^V$ to denote the indicator of 
$I$, that is, 
\[
 \one_I(v)=\begin{cases}
	     1 &\text{if }v\in I,\\
	     0 &\text{if }v\not\in I,
	    \end{cases}
\]
where $v\in V$.
We identify $I$ and $\one_I$, and we write $v\in\one_I$ to mean $v\in I$. 
For simplicity we just write $\one$ to mean $\one_V$, so $\one(v)=1$
for all $v\in V$.
For $x,y\in V$ and a matrix $T$, where the rows and columns are indexed by $V$,
we write $(T)_{x,y}$ for the $(x,y)$-entry of $T$.
Throughout the paper let $q:=1-p$ and $q_i:=1-p_i$.

\subsection{Hoffman's bound for a graph}
Let $V$ be a finite set with $|V|\geq 2$. We say that $\mu_2:V\times V\to\R$
is a symmetric signed measure if 
\begin{itemize}
\item $\mu_2(x,y)=\mu_2(y,x)$ for all $x,y\in V$, and
\item $\sum_{x\in V}\sum_{y\in V}\mu_2(x,y)=1$.
\end{itemize}
Note that $\mu_2(x,y)$ can be negative, which is essential for the 
proof Theorem~\ref{thm1}.
Let $\mu_1:V\to\R$ be the marginal of $\mu_2$, that is,
\begin{align}\label{eq2:mu1}
 \mu_1(x):=\sum_{y\in V}\mu_2(x,y). 
\end{align}
Then $\sum_{x\in V}\mu_1(x)=\sum_{x\in V}\sum_{y\in V}\mu_2(x,y)=1$.
\begin{defn}
Let $\mu_2:V\times V\to\R$ be a symmetric signed measure.
We say that $G=(V,\mu_2)$ a weighted graph if
\begin{align}\label{positivity}
\mu_1(x)>0 \text{ for all }  x\in V. 
\end{align}
\end{defn}
In this paper we only deal with $\mu_2$ whose marginal $\mu_1$ satisfies
\eqref{positivity}.
Here $V$ is the vertex set, and each
$(x,y)\in V\times V$ is an ordered edge (or a directed edge) 
including loops with possibly negative weight $\mu_2(x,y)$. 
So $(x,y)$ should be considered a non-edge if and only if $\mu_2(x,y)=0$.
We say that $I\subset V$ is an \emph{independent set} if $x,y\in I$ implies 
$\mu_2(x,y)=0$. 
Write $\mu_1(I)$ for $\sum_{x\in I}\mu_1(x)$, and define the independence
ratio $\alpha(G)$ by
\[
\alpha(G):=\max\{\mu_1(I):\text{$I$ is an independent set in $G$}\}. 
\]

\begin{example}
Let $G=(V,E)$ be a usual simple $d$-regular graph. 
Let us construct a symmetric measure $\mu_2$ so that $\mu_1$ becomes 
a uniform measure $\mu_1(x)\equiv 1/|V|$. To this end we just set
\[
 \mu_2(x,y):=\begin{cases}
	      0 & \text{if }\{x,y\}\not\in E\\
	      \frac1{d|V|} & \text{if }\{x,y\}\in E.
	     \end{cases}
\]
Then $(V,\mu_2)$ is a weighted graph, and in this case $\alpha(G)=|I|/|V|$,
where $I\subset V$ is a usual maximum independent set in $G$.
\qed
\end{example}

Let $\R^V$ be the set of functions from $V$ to $\R$, and for $f,g\in\R^V$ let
\begin{align*}
 \E_{\mu_1}[f]&:=\sum_{x\in V}f(x)\mu_1(x),\\ 
 \E_{\mu_2}[f,g]&:=\sum_{x\in V}\sum_{y\in V}f(x)g(y)\mu_2(x,y).
\end{align*}
Since $\mu_2$ is symmetric it follows $\E_{\mu_2}[f,g]=\E_{\mu_2}[g,f]$.

\begin{fact}\label{fact0}
Let $\phi:=\one_I\in\{0,1\}^V$ be the indicator of an independent set $I$.
Then we have $\E_{\mu_1}[\phi]=\mu_1(I)$ and $\E_{\mu_2}[\phi,\phi]=0$.
\end{fact}
\begin{proof}
Indeed we have
\[
\E_{\mu_1}[\phi]=\sum_{x\in V}\phi(x)\mu_1(x)=\sum_{x\in I}\mu_1(x)=\mu_1(I). 
\]
Since $\mu_2(x,y)=0$ for $x,y\in I$ we also have
\[
\E_{\mu_2}[\phi,\phi]=\sum_{x\in V}\sum_{y\in V}\phi(x)\phi(y)\mu_2(x,y)
=\sum_{x\in I}\sum_{y\in I}\mu_2(x,y)=0. 
\]
\end{proof}

We define a measure version of the adjacency matrix, which is an extension
of the usual adjacency matrix, and we simply call it an adjacency matrix
in this paper.

\begin{defn}
Let $G=(V,\mu_2)$ be a weighted graph.
We define the adjacency matrix $T=T(G)$. This is a $|V|\times|V|$ matrix,
and for $x,y\in V$ the $(x,y)$-entry of $T$ is given by
\begin{align}\label{def:T}
 (T)_{x,y}=\frac{\mu_2(x,y)}{\mu_1(x)}. 
\end{align}
\end{defn}
We introduce an inner product 
$\la \cdot,\cdot\ra_{\mu_1}:\R^V\times\R^V\to\R$ by
\begin{align}\label{inner product}
 \la f,g\ra_{\mu_1}:=\E_{\mu_1}[fg]=\sum_{x\in V}f(x)g(x)\mu_1(x). 
\end{align}
Note that the condition \eqref{positivity} is necessary to define the above
inner product properly.
Clearly $\la f,g\ra_{\mu_1}=\la g,f\ra_{\mu_1}$. 
We list some easy facts. (We include the proof in Appendix.)
\begin{fact}\label{fact1}
Let $G=(V,\mu_2)$ be a weighted graph with the adjacency matrix $T$.
Let $f,g\in\R^V$ and $\phi\in\{0,1\}^V$.
\begin{itemize}
\item[(i)] $\la f,Tg\ra_{\mu_1}=\E_{\mu_2}[f,g]$.
\item[(ii)] $\la f,Tg\ra_{\mu_1}=\la Tf,g\ra_{\mu_1}$,
that is, $T$ is self-adjoint.
\item[(iii)] $T\one=\one$.
\item[(iv)] $\la\one,\one\ra_{\mu_1}=1$.
\item[(v)] $\la \phi,\one\ra_{\mu_1}=\E_{\mu_1}[\phi]$.
\item[(vi)] $\la \phi,\phi\ra_{\mu_1}=\E_{\mu_1}[\phi]$.
\end{itemize}
\end{fact}
\begin{setup}\label{setup}
Let $G=(V,\mu_2)$ be a weighted graph with the adjacency matrix $T$.
By Fact~\ref{fact1} (iii) the matrix $T$ has eigenvector $\one$ with the 
eigenvalue $1$. Since $T$ is self-adjoint, $T$ has $|V|$ real eigenvalues
$l_0=1,l_1,l_2,\ldots,l_{|V|-1}$ with the corresponding eigenvectors
$\vv_0=\one,\vv_1,\vv_2\ldots,\vv_{|V|-1}$,
that is, $T\vv_i=l_i\vv_i$.
We may assume that these vectors consist of an orthonormal basis of $\R^V$
with respect to the inner product $\la\cdot,\cdot\ra_{\mu_1}$.
Then, for $\phi\in\{0,1\}^V$, we can expand $\phi$ using the basis:
\begin{align}\label{eq:setup}
 \phi=\hat\phi_0\one+\sum_{i\geq 1}\hat\phi_i\vv_i, 
\end{align}
where $\hat\phi_i=\la \phi,\vv_i\ra_{\mu_1}$.
Let $\lambda_{\min}(T)$ denote the minimum eigenvalue of $T$. \qed
\end{setup}

\begin{fact}\label{fact2}
Let $\phi\in\{0,1\}^V$. 
\begin{itemize}
\item[(i)] $\E_{\mu_1}[\phi]=\hat\phi_0$ and 
$\E_{\mu_1}[\phi]=\hat\phi_0^2+\sum_{i\geq 1}\hat\phi_i^2$.
\item[(ii)] $\E_{\mu_2}[\phi,\phi]=\hat\phi_0^2+\sum_{i\geq 1}\hat\phi_i^2l_i$.
\end{itemize}
\end{fact}

\begin{lemma}\label{E[f,f]>}
For $\phi\in\{0,1\}^V$ we have
$\E_{\mu_2}[\phi,\phi]\geq
\E_{\mu_1}[\phi]\big(1-(1-\lambda_{\min}(T))(1-\E_{\mu_1}[\phi])\big)$.
\end{lemma}
\begin{proof}
By Fact~\ref{fact2} (ii) we have
\begin{align*}
\E_{\mu_2}[\phi,\phi] &=
\hat\phi_0^2+\sum_{i\geq 1} \hat\phi_i^2l_i\\
&\geq \hat\phi_0^2+\lambda_{\min}(T)\sum_{i\geq 1} \hat\phi_i^2\\
&= \E_{\mu_1}[\phi]^2+\lambda_{\min}(T)(\E_{\mu_1}[\phi]-
\E_{\mu_1}[\phi]^2)\quad\text{by Fact~\ref{fact2} (i)}\\
&= \E_{\mu_1}[\phi]\big(1-(1-\lambda_{\min}(T))(1-\E_{\mu_1}[\phi])\big).
\end{align*}
\end{proof}

\begin{thm}[Hoffman's bound, see \cite{Haemers}]\label{thm:Hoffman bound for 2-graph}
Let $G=(V,\mu_2)$ be a weighted graph with the adjacency matrix $T$.
Let $\phi$ be the indicator of an independent set of $G$.
Suppose that $\lambda_{\min}(T)<1$. Then we have
\[
1-\E_{\mu_1}[\phi]\geq\frac1{1-\lambda_{\min}(T)},
\]
and 
\[
 \alpha(G)\leq\frac{-\lambda_{\min}(T)}{1-\lambda_{\min}(T)}.
\]
\end{thm}
\begin{proof}
By Lemma~\ref{E[f,f]>} with $\E_{\mu_2}[\phi,\phi]=0$ it follows
\[
0\geq \E_{\mu_1}[\phi]\big(1-(1-\lambda_{\min}(T))(1-\E_{\mu_1}[\phi])\big)
=\E_{\mu_1}[\phi]\big(\lambda_{\min}(T)+(1-\lambda_{\min}(T))\E_{\mu_1}[\phi]\big).
\]
Since $\E_{\mu_1}(\phi)>0$ and $\lambda_{\min}(T)<1$ we get the desired
inequality.
\end{proof}

\subsection{Hoffman's bound for a 3-graph}
Let $V$ be a finite set with $|V|\geq 2$, 
and let $\mu_3:V^3\to\R$ be a symmetric signed measure, that is,
\begin{itemize}
\item $\mu_3(x,y,z)=\mu_3(p,q,r)$ whenever
$(p,q,r)$ is a permutation of $(x,y,z)$, and
\item $\sum_{x\in V}\sum_{y\in V}\sum_{z\in V}\mu_3(x,y,z)=1$.
\end{itemize}
Define the marginals $\mu_2\in\R^{V^2}$ and $\mu_1\in\R^V$ as follows:
\begin{align*}
\mu_2(x,y)&:= \sum_{z\in V}\mu_3(x,y,z).\\
\mu_1(x)&:= \sum_{y\in V}\mu_2(x,y)=\sum_{y\in V}\sum_{z\in V}\mu_3(x,y,z).
\end{align*}
\begin{defn}
Let $\mu_3:V^3\to\R$ be a symmetric signed measure.
We say that $H=(V,\mu_3)$ is a weighted 3-graph if
$\mu_2(x,y)>0$ for all $x,y\in V$.
\end{defn}
Note that $\mu_1(x)>0$ for all $x$ follows from \eqref{eq2:mu1}.
We will consider two inner products, one is with respect to $\mu_1(x)$, and
the other is with respect to $\mu_2(x,y)/\mu_1(x)$ for fixed $x$. 
We need the conditions in the above 
definition to ensure that these inner products are defined properly.

We say that a subset $I\subset V$ is an independent set in $H$ if
$\mu_3(x,y,z)=0$ for all $x,y,z\in I$, and we define the independence
ratio $\alpha(H)$ by
\[
 \alpha(H):=\max\{\mu_1(I):\text{$I$ is an independent set in $H$}\}. 
\]

Suppose that $H=(V,\mu_3)$ is a weighted 3-graph.
Then $G=(V,\mu_2)$ is a weighted 2-graph because $\mu_2$ is a symmetric
measure and $\mu_1$ satisfies \eqref{positivity}.
The adjacency matrix $T=T(G)$ is defined by \eqref{def:T}.
Now we define a link graph relative to $x$ which will be denoted by $H_x$.
To this end, for $x,y,z\in V$, we define 
$\mu_{2,x}\in\R^{V^2}$ and its marginal $\mu_{1,x}\in\R^V$ by
\begin{align*}
 \mu_{2,x}(y,z)&:=\frac{\mu_3(x,y,z)}{\mu_1(x)},\\
 \mu_{1,x}(y)&:=\sum_{z\in V}\mu_{2,x}(y,z)=\frac{\mu_2(x,y)}{\mu_1(x)}.
\end{align*}
Then $H_x:=(V,\mu_{2,x})$ is a weighted 2-graph because $\mu_{2,x}$ is a 
symmetric measure and $\mu_{1,x}$ satisfies \eqref{positivity}.
The adjacency matrix $T_x=T_x(H_x)$ is also defined by \eqref{def:T},
so the $(y,z)$-entry of $T_x$ is
\begin{align}\label{def:Tx}
(T_x)_{y,z}=\frac{\mu_{2,x}(y,z)}{\mu_{1,x}(y)}=\frac{\mu_3(x,y,z)}{\mu_2(x,y)}.\end{align}
By definition both $T$ and $T_x$ are self-adjoint, and they have $|V|$ real eigenvalues.

We can relate $\E_{\mu_2}$ and $\E_{\mu_{1,x}}$ as follows.
Here we write $x\in\phi$ to mean $\phi(x)=1$.
\begin{lemma}\label{claim:E[f,f]}
For $\phi\in\{0,1\}^V$ we have
$\E_{\mu_2}[\phi,\phi] 
\leq \E_{\mu_1}[\phi]\,\max_{x\in\phi}\E_{\mu_{1,x}}[\phi]$.
\end{lemma}
\begin{proof}
Note that if $x\not\in\phi$ then $\phi(x)=0$ and the term having $\phi(x)$ 
does not contribute in the sum below. Thus we have
\begin{align*}
\E_{\mu_2}[\phi,\phi] &= \sum_{x\in V}\sum_{y\in V} \phi(x)\phi(y)\mu_2(x,y)\\
&= \sum_{x\in\phi}\phi(x)\mu_1(x)\sum_{y\in V} 
\phi(y)\frac{\mu_2(x,y)}{\mu_1(x)}\\
&= \sum_{x\in\phi}\phi(x)\mu_1(x)\sum_{y\in V} \phi(y)\mu_{1,x}(y)\\
&\leq\sum_{x\in V}\phi(x)\mu_1(x)\,\max_{x\in\phi}\sum_{y\in V}\phi(y)\mu_{1,x}(y)\\
&= \E_{\mu_1}[\phi]\,\max_{x\in\phi}\E_{\mu_{1,x}}[\phi].
\end{align*} 
\end{proof}

\begin{thm}[Hoffman's bound for a 3-graph \cite{FGL}]\label{thm:3-graph Hoffman}
Let $H=(V,\mu_3)$ be a weighted $3$-graph. Let $\phi$ be the indicator 
of an independent set. Suppose that $\lambda_{\min}(T)<1$ and 
$\lambda_{\min}(T_x)<1$ for $x\in\phi$. Then
\[
 1-\E_{\mu_1}[\phi]\geq\frac1{(1-\lambda_{\min}(T))\max_{x\in\phi}(1-\lambda_{\min}(T_x))}.
\]
In particular, if $\phi$ is the indicator of a maximum independent set, then
\[
 \alpha(H)\leq 1-
\frac1{(1-\lambda_{\min}(T))\max_{x\in\phi}(1-\lambda_{\min}(T_x))}.
\]
\end{thm}
\begin{proof}
Let $I\subset V$ be the independent set in $H$ such that $\one_I=\phi$.
By Lemma~\ref{E[f,f]>} we have
\begin{align*}
\E_{\mu_2}[\phi,\phi]\geq\E_{\mu_1}[\phi]\big(
1-(1-\lambda_{\min}(T))(1-\E_{\mu_1}[\phi])\big).
\end{align*}
This together with Lemma~\ref{claim:E[f,f]} yields
\[
\E_{\mu_1}[\phi]\,\max_{x\in\phi}\E_{\mu_{1,x}}[\phi]\geq
\E_{\mu_1}[\phi]\big(1-(1-\lambda_{\min}(T))(1-\E_{\mu_1}[\phi])\big),
\]
that is,
\begin{align}\label{eq:la1+la2}
 1-\E_{\mu_1}[\phi]\geq\frac{1-\max_{x\in\phi}\E_{\mu_{1,x}}[\phi]}{1-\lambda_{\min}(T)}.
\end{align}

Next we bound $\E_{\mu_1,x}[\phi]$ by using the 
the link graph $H_x:=(V,\mu_{2,x})$ relative to $x\in I$. 
Note that $I$ is an independent set in $H_x$ as well.
Indeed if $y,z\in I$ then $\mu_{2,x}(y,z)=0$ because $\mu_3(x,y,z)=0$.
By applying Theorem~\ref{thm:Hoffman bound for 2-graph} to 
the adjacency matrix $T_x$ of $H_x$ we get 
\[
1-\E_{\mu_{1,x}}[\phi]\geq\frac1{1-\lambda_{\min}(T_x)},   
\]
and
\begin{align}\label{eq:1-E[mu_{1,x}]}
 1-\max_{x\in\phi}\E_{\mu_{1,x}}[\phi]\geq
\frac1{\max_{x\in\phi}(1-\lambda_{\min}(T_x))}.  
\end{align}
By \eqref{eq:la1+la2} and \eqref{eq:1-E[mu_{1,x}]}
we obtain the desired inequality.
\end{proof}

\subsection{Tools for uniqueness}\label{subsec:unique}
In Setup~\ref{setup} any $\phi\in\{0,1\}^V$ can be expanded in the form
in \eqref{eq:setup}. We first show that if $\E_{\mu_2}[\phi,\phi]$
is small, then we only need the eigenvectors corresponding to the largest
and the smallest eigenvalues for the expansion.

\begin{lemma}\label{reduction1}
We assume Setup~\ref{setup}. If 
\begin{align}\label{eq:red1}
\E_{\mu_2}[\phi,\phi]\leq
\hat\phi_0^2+\lambda_{\min}(T)(\hat\phi_0-\hat\phi_0^2),
\end{align}
then $\phi=\hat\phi_0\one+\sum_{i\in J}\hat\phi_i\vv_i$, where
$J=\{i:1\leq i<|V|,\,l_i=\lambda_{\min}(T)\}$.
\end{lemma}

\begin{proof}
By Fact~\ref{fact2} and \eqref{eq:red1} we have
\[
\E_{\mu_2}[\phi,\phi]=\hat\phi_0^2+\sum_{i\geq 1}\hat\phi_i^2l_i
\leq\hat\phi_0^2+\lambda_{\min}(T)(\hat\phi_0-\hat\phi_0^2)
=\hat\phi_0^2+\lambda_{\min}(T)\sum_{i\geq 1}\hat\phi_i^2,
\]
and 
\[
 \sum_{i\geq 1}(l_i-\lambda_{\min}(T))\hat\phi_i^2 \leq 0.
\]
This yields $l_i=\lambda_{\min}(T)$ or $\hat\phi_i=0$, and the result follows.
\end{proof}

Let $1>p_1\geq p_2\geq\cdots\geq p_n>0$ be given. Let $V=2^{[n]}$ and let
$\mu_1:V\to(0,1)$ be the measure defined by 
$\mu_1(S)=\prod_{i\in S}p_i\prod_{j\in [n]\setminus S}q_j$ for $S\in V$.
Then we can view $\R^V$ as a $2^n$-dimensional inner space over $\R$,
where the inner product is defined by \eqref{inner product}.
We will construct an orthonormal basis that suits our purpose.

\begin{fact}\label{T^(i)}
Let $p_i>\frac{r-2}{r-1}$ and let
\[
 T^{(i)}=
\left[
\begin{matrix}
    1-\frac{p_i}{(r-1)q_i}&    \frac{p_i}{(r-1)q_i}\\
    \frac1{r-1}& 1-\frac1{r-1}
\end{matrix}
\right].
\]
Then $T^{(i)}$ has eigenvalues 1 and $\lambda_i:=1-\frac{1}{(r-1)q_i}<0$ 
with the corresponding eigenvectors 
$\vv_{\emptyset}^{(i)}=\left(\begin{smallmatrix}    1\\1   \end{smallmatrix}\right)$
and $\vv_{\{i\}}^{(i)}=\left(\begin{smallmatrix}    c_i\\-\frac1{c_i}   \end{smallmatrix}\right)$,
where $c_i=\sqrt{p_i/q_i}$. 
\end{fact}
Let $T=T^{(n)}\otimes T^{(n-1)}\otimes\cdots\otimes T^{(1)}$.
Then the rows and columns of $T$ are indexed by the order
$\emptyset, \{1\}, \{2\}, \{1,2\},\{3\},\{1,3\},\{2,3\},\{1,2,3\},\ldots$. 
For each $S\in V$ the corresponding indicator is given
by the column vector of the matrix
\[
\left[
\begin{matrix}
    1&0\\
    1&1
\end{matrix}
\right]
\otimes\cdots\otimes
\left[
\begin{matrix}
    1&0\\
    1&1
\end{matrix}
\right].
\]
One can construct the eigenvectors of $T$ by routine computation, 
and we have the following, see e.g., \cite{Friedgut, STT}, 
\begin{fact}\label{ONB}
Let $S\in V$ and let $\vv_S$ be the column vector (indexed by $S$) of
the matrix
\[
C_n:=
\left[
\begin{matrix}
    1&c_n\\
    1& -\frac1{c_n}
\end{matrix}
\right]
\otimes
\left[
\begin{matrix}
    1&c_{n-1}\\
    1& -\frac1{c_{n-1}}
\end{matrix}
\right]
\otimes\cdots\otimes
\left[
\begin{matrix}
    1&c_1\\
    1& -\frac1{c_1}
\end{matrix}
\right].
\]
\begin{itemize}
 \item[(i)] 
The $\R^V$ with the inner product defined by \eqref{inner product}
is spanned by the orthonormal basis $\{\vv_S:S\in V\}$.
\item[(ii)]
The $\vv_S$ is an eigenvector of $T$ with the corresponding eigenvalue 
$\lambda_S:=\prod_{j\in S}\lambda_j$.
In particular, $\vv_{\emptyset}=\one$ and $\lambda_{\emptyset}=1$.
\item[(iii)]
The entry of $\vv_{\{i\}}$ corresponding to $S\in V$ is
$c_i$ if $i\not\in S$ and $-1/c_i$ if $i\in S$, and
$p_i\one-\sqrt{p_iq_i}\vv_{\{i\}}$ is the indicator of the 
star centered at $i$, i.e., $\{S\in V:i\in S\}$.
\item[(iv)]
We have $\vv_S=\prod_{i\in S}\vv_{\{i\}}$, where the product is 
taken componentwise.
\end{itemize}
\end{fact}
For example, the matrix $C_3$ is as follows, where the columns
are in the order $\vv_\emptyset$, $\vv_{\{1\}}$, $\vv_{\{2\}}$,
$\vv_{\{1,2\}}$, $\vv_{\{3\}}$, $\vv_{\{1,3\}}$, $\vv_{\{2,3\}}$,
$\vv_{\{1,2,3\}}$.
\[
C_3=
\left[
\begin{array}{cccccccc}
 1 & c_1 & c_2 & c_1 c_2 &
   c_3 & c_1 c_3 & c_2
   c_3 & c_1 c_2 c_3 \\
 1 & -\frac{1}{c_1} & c_2 &
   -\frac{c_2}{c_1} & c_3 &
   -\frac{c_3}{c_1} & c_2
   c_3 & -\frac{c_2
   c_3}{c_1} \\
 1 & c_1 & -\frac{1}{c_2} &
   -\frac{c_1}{c_2} & c_3 &
   c_1 c_3 &
   -\frac{c_3}{c_2} &
   -\frac{c_1 c_3}{c_2} \\
 1 & -\frac{1}{c_1} &
   -\frac{1}{c_2} &
   \frac{1}{c_1 c_2} & c_3
   & -\frac{c_3}{c_1} &
   -\frac{c_3}{c_2} &
   \frac{c_3}{c_1 c_2} \\
 1 & c_1 & c_2 & c_1 c_2 &
   -\frac{1}{c_3} &
   -\frac{c_1}{c_3} &
   -\frac{c_2}{c_3} &
   -\frac{c_1 c_2}{c_3} \\
 1 & -\frac{1}{c_1} & c_2 &
   -\frac{c_2}{c_1} &
   -\frac{1}{c_3} &
   \frac{1}{c_1 c_3} &
   -\frac{c_2}{c_3} &
   \frac{c_2}{c_1 c_3} \\
 1 & c_1 & -\frac{1}{c_2} &
   -\frac{c_1}{c_2} &
   -\frac{1}{c_3} &
   -\frac{c_1}{c_3} &
   \frac{1}{c_2 c_3} &
   \frac{c_1}{c_2 c_3} \\
 1 & -\frac{1}{c_1} &
   -\frac{1}{c_2} &
   \frac{1}{c_1 c_2} &
   -\frac{1}{c_3} &
   \frac{1}{c_1 c_3} &
   \frac{1}{c_2 c_3} &
   -\frac{1}{c_1 c_2 c_3}
   \\
\end{array}
\right].
\]

\begin{lemma}\label{reduction2}
Let $L=\{i\in[n]:p_i=p_1\}$ and $\phi\in\{0,1\}^V$.
Suppose that $\lambda_{\min}(T)<\lambda_S$ for all $S\in V\setminus\binom L1$.
If $\phi(\emptyset)=0$ and $\phi([n])=1$, and $\phi$ is expanded as
\begin{align}\label{unique expansion}
  \phi=p_1\one+\sum_{k\in L}\hat\phi_{\{k\}}\vv_{\{k\}},  
\end{align}
then $\phi$ is the indicator of a star centered at some $i\in L$.
\end{lemma}

\begin{proof}
We first show that there is only one $i\in [n]$ such that
$\phi=p_1\one-\sqrt{p_iq_i}\,\vv_{\{i\}}$.
Suppose, to the contrary, that there are distinct $i,j$ such that both
$\hat\phi_{\{i\}}$ and $\hat\phi_{\{j\}}$ are non-zero.
Let $\phi^2\in\{0,1\}^V$ be such that $\phi^2(x)=\phi(x)^2$.
Then, by Fact~\ref{ONB} (iv), 
$\phi^2=(p_1\one+\sum_{k\in L}\hat\phi_{\{k\}}\vv_{\{k\}})^2$ must 
contain the term 
\[
\hat\phi_{\{i\}}\hat\phi_{\{j\}}\vv_{\{i\}}\vv_{\{j\}} 
=\hat\phi_{\{i\}}\hat\phi_{\{j\}}\vv_{\{i,j\}}
\]
whose coefficient $\hat\phi_{\{i\}}\hat\phi_{\{j\}}$ is non-zero.
But this contradicts the fact that the expansion \eqref{unique expansion}
is unique and $\phi=\phi^2$.

Therefore we can write $\phi=p_1\one+\hat\phi_{\{i\}}\vv_{\{i\}}$ for some 
$i\in[n]$. By Fact~\ref{ONB} (iii) we have
$\vv_{\{i\}}(\emptyset)=c_i$ and $\vv_{\{i\}}([n])=-1/c_i$, and so
\[
\phi(\emptyset)=0=p_1+\hat\phi_{\{i\}}c_i \text{ and }
\phi([n])=1=p_1-\hat\phi_{\{i\}}/c_i.
\]
Solving the equations we get $\hat\phi_{\{i\}}=-\sqrt{p_1q_1}$ and 
$c_i=c_1$. This means that $i\in L$.
Consequently $\phi=p_i\one-\sqrt{p_iq_i}\vv_{\{i\}}$, and by (iii) of 
Fact~\ref{ONB} we complete the proof.
\end{proof}

\section{Application}
Recall that a family of subsets $\AA\subset 2^{[n]}$ is $r$-wise intersecting
if $A_1\cap A_2\cap\cdots\cap A_r\neq\emptyset$ for all 
$A_1,\ldots,A_r\in\AA$.
Let $\pp=(p_1,\ldots,p_n)\in(0,1)^n$ be a fixed real vector.
The $\mu_{\pp}$-measure of a family $\AA\subset 2^{[n]}$ is defined by
\[
 \mu_{\pp}(\AA):=\sum_{A\in\AA}\prod_{i\in A}p_i\prod_{j\in[n]\setminus A}q_j.
\]

\subsection{2-wise case: Proof of Theorem~\ref{thm1}}
\begin{proof}[Proof of Theorem~\ref{thm1}]
The case $n=1$. In this case the only intersecting family is $\AA=\{\{1\}\}$,
and $\mu_{\pp}(\AA)=p_1$, where $\pp=(p_1)$.
But we will get this result by using Hoffman's bound because 
the spectral information in this case will be used later to get the 
spectral information for the general $n\geq 2$ case.

Let $V^{(1)}=2^{\{1\}}=\{\emptyset,\{1\}\}$, and define the 
symmetric signed measure $\mu_2^{(1)}:V^{(1)}\times V^{(1)}\to\R$ by 
\[
 \mu_2^{(1)}(\emptyset,\{1\})= \mu_2^{(1)}(\{1\},\emptyset)=p_1,\quad
 \mu_2^{(1)}(\emptyset,\emptyset)=1-2p_1,\quad
 \mu_2^{(1)}(\{1\},\{1\})=0.
\]
This induces the marginal
\[
\mu_1^{(1)}(\{1\})=p_1,\quad\mu_1^{(1)}(\emptyset)=q_1.
\]
Then we obtain a weighted 2-graph $G=(V,\mu_2^{(1)})$. 
Note that $\mu_{\pp}=\mu_1^{(1)}$.
(Indeed this $\mu_2$ is the only symmetric signed measure which satisfies
$\mu_2^{(1)}(\{1\},\{1\})=0$ and $\mu_{\pp}=\mu_1^{(1)}$.)
The adjacency matrix $T^{(1)}$ is given by
\[
 T^{(1)}=\left[
\begin{matrix}
    1-\frac{p_1}{q_1}&    \frac{p_1}{q_1}\\
    1& 0
\end{matrix}
\right],
\]
where the rows and columns are indexed in the order $\emptyset,\{1\}$.
This matrix has eigenvalues $1$ and $-\frac{p_1}{q_1}$. Thus 
$\lambda_{\min}(T^{(1)})=-\frac{p_1}{q_1}$.
Then by Theorem~\ref{thm:Hoffman bound for 2-graph} we have
\begin{align*}
1- \alpha(G)&\geq 
\frac1{1-\lambda_{\min}(T^{(1)})}=\frac1{1+\frac{p_1}{q_1}}=q_1,
\end{align*}
and $\alpha(G)\leq 1-q_1=p_1$.
Now it follows from the definition of $\mu_2^{(1)}$ that a 2-wise intersecting 
family $\AA\subset V^{(1)}$ is an independent set in $G$. 
Thus we have shown that $\mu_\pp(\AA)\leq p_1$ in this case $n=1$.

\medskip
The general case $n\geq 2$. 
For $i=1,2,\ldots, n$, let $V_i=2^{\{i\}}$ and 
let $\mu_2^{(i)}$ be defined as in the previous $n=1$ case.
Let $G^{(i)}=(V^{(i)},\mu_2^{(i)})$ with the adjacency matrix
$T^{(i)}$, where
\[
 T^{(i)}=
\left[
\begin{matrix}
    1-\frac{p_i}{q_i}&    \frac{p_i}{q_i}\\
    1& 0
\end{matrix}
\right].
\]
Now we define $G=(V,\mu_2)$ to be a product of $G^{(1)},\ldots, G^{(n)}$.
To this end let $V=V^{(1)}\times \cdots\times V^{(n)}\cong 2^{[n]}$, 
and define $\mu_2:V^2\to\R$ by 
$\mu_2=\mu_2^{(1)}\times\cdots\times\mu_2^{(n)}$, that is, for $S,S'\in V$, let
\[
\mu_2(S,S')
:=\mu_2^{(1)}(s_1,s'_1)\times\cdots\times\mu_2^{(n)}(s_n,s_n'),
\]
where $s_i=S\cap\{i\}$ and $s_i'=S'\cap\{i\}$ for $1\leq i\leq n$.
\begin{claim}\label{mu1=mu1^1...}
The marginal $\mu_1$ satisfies $\mu_1=\mu_1^{(1)}\times\cdots\times\mu_1^{(n)}$,
and $\mu_1=\mu_\pp$. 
\end{claim}
\begin{proof}
Indeed, by \eqref{eq2:mu1}, we have
\begin{align*}
\mu_1(S)
&=\sum_{S'\in V}\mu_2(S,S')\\
&=\sum_{S'\in V}\mu_2^{(1)}(s_1,s'_1)\times\cdots\times\mu_2^{(n)}(s_n,s_n')\\
&=\sum_{s_1'\in V^{(1)}}\mu_2^{(1)}(s_1,s_1')\times\cdots\times
\sum_{s_n'\in V^{(n)}}\mu_2^{(n)}(s_n,s_n')\\
&=\mu_1^{(1)}(s_1)\times\cdots\times\mu_1^{(n)}(s_n) =\mu_\pp(S).
\end{align*}
\end{proof}
Thus $0<\mu_1(S)<1$ for all $S\in V$, and $G$ is a weighted graph.
By construction we see that the adjacency matrix is 
$T=T^{(n)}\otimes\cdots\otimes T^{(1)}$.
We can apply Fact~\ref{T^(i)} with $r=2$ and Fact~\ref{ONB}.
Then, for each $S\in V$, $T$ has an eigenvalue 
$\lambda_S:=\prod_{j\in S}\left(-\frac{p_j}{q_j}\right)$
with the corresponding eigenvector $\vv_S$.
Now we determine $\lambda_{\min}(T)=\min_{S}\lambda_S$.

\begin{claim}
We have $\lambda_{\min}(T)=\lambda_{\{1\}}=-\frac{p_1}{q_1}$,
and if $\lambda_S=\lambda_{\min}(T)$ then $S=\{i\}$ with $p_i=p_1$.
\end{claim}
\begin{proof}
Since $p_i\geq p_{i+1}$ we have
$\lambda_{\{i\}}=-\frac{p_i}{q_i}\leq -\frac{p_{i+1}}{q_{i+1}}
=\lambda_{\{i+1\}}<0$, and $\min_i\lambda_{\{i\}}=\lambda_{\{1\}}$.
The assumption $p_3<\frac12$ means $-1<\lambda_{\{3\}}$, 
and so $-1<\lambda_{\{i\}}<0$ for all $3\leq i\leq n$.
Thus if $\lambda_{\min}(T)=\lambda_S$ then, using Fact~\ref{ONB} (ii), 
$S$ contains at most one $i$ with $i\geq 3$. In particular if 
$\lambda_S=\lambda_{\{1\}}$ then $S=\{i\}$ with $p_i=p_1$.

If $p_1\leq\frac12$ then $-1\leq\lambda_{\{1\}}$, and 
$-1\leq\lambda_{\{1\}}\leq\lambda_{\{2\}}<0$. 
Therefore if $i\in S\subset[n]$ then
$\lambda_{\{i\}}\leq\lambda_S$ with equality holding if and only if
$S=\{j\}$ with $p_j=p_i$. Thus we get the statement of the claim in this case.

If $p_1>\frac12$ then $\lambda_{\{1\}}<-1$. Thus we have
$\lambda_{\min}(T)=\min\{\lambda_{\{1\}},\lambda_{\{1,2,3\}}\}$.
By simple computation we see that
$\lambda_{\{1\}}<\lambda_{\{1,2,3\}}$ is equivalent to $p_3<q_2$, 
which is our assumption. Thus we get the statement of the claim again.
\end{proof}
Thus by Theorem~\ref{thm:Hoffman bound for 2-graph} we have $\alpha(G)\leq p_1$.
Now let $\AA\subset 2^{[n]}$ be a $2$-wise intersecting family. 
If $A,B\in\AA$ then there is some $i\in A\cap B$. 
Then $\mu_2(A,B)=0$ follows from the fact that $\mu_2^{(i)}(\{i\},\{i\})=0$ with
the definition of $\mu_2$. This means that $\AA\subset V$ is an independent set
of $G$. Since $\mu_1=\mu_\pp$ we see that 
$\mu_\pp(\AA)=\mu_1(\AA)\leq\alpha(G)\leq p_1$,
which completes the proof of inequality.

Finally we prove the uniqueness.
Suppose that $\alpha(G)=p_1$ and let $\phi$ be the indicator of
a maximum independent set. 
Then $\hat\phi_\emptyset=\la \phi,\one\ra_{\mu_1}=\E_{\mu_1}[\phi]=p_1$.
We also have $\E_{\mu_2}[\phi,\phi]=0$ by Fact~\ref{fact0}. Thus, using
$\hat\phi_\emptyset=p_1$ and $\lambda_{\min}(T)=-\frac{p_1}{q_1}$, 
we can verify \eqref{eq:red1}, and by Lemma~\ref{reduction1} we have
$\phi=p_1\one+\sum_{S\in W}\hat\phi_S\vv_S$, where
$W=\{S\in V:\lambda_S=\lambda_{\min}(T)\}$.
Since $\lambda_S=\lambda_{\min}(T)$ is equivalent to 
$S=\{i\}$ with $p_i=p_1$, we can rewrite 
$\phi=p_1\one+\sum_{k\in L}\hat\phi_{\{k\}}\vv_{\{k\}}$, where
$L=\{i\in[n]:p_i=p_1\}$.
Consequently it follows from Lemma~\ref{reduction2} that $\phi$ is the 
indicator of a star centered at some $i\in L$. 
This completes the proof of Theorem~\ref{thm1}.
\end{proof}

\begin{example}
Define a 2-wise intersecting family $\AA\subset 2^{[n]}$ by 
$\AA=\{A\in 2^{[n]}:|A\cap[3]|\geq 2\}$, and let
$\pp=(p_1,p_2,\ldots,p_n)$. Then 
\[
 \mu_\pp(\AA) = p_1p_2q_3+p_1q_2p_3+q_1p_2p_3+p_1p_2p_3
=p_1p_2+p_1p_3+p_2p_3-2p_1p_2p_3.
\]
If $p_1=p_2=p_3$ then $\mu_\pp(\AA)=p_1^2(3-2p_1)$ and
$\mu_\pp(\AA)>p_1$ iff $\frac12<p_1<1$.
If $p_1=p_2$ and $p_3=\frac12$ then $\mu_\pp(\AA)=p_1$. 

These examples show the sharpness of the condition $p_3<\frac12$ in
Conjecture~\ref{conj1} (if true) in the following sense. First, for the
inequality ($\mu_\pp(\AA)\leq p_1$) we cannot 
replace the condition with $p_4<\frac12$. Second, to ensure the uniqueness 
we cannot replace the condition with $p_3\leq\frac12$. 
\qed
\end{example}

\subsection{3-wise case: Proof of Theorem~\ref{thm2}}
\begin{prop}\label{prop1}
Let $\frac23>p_1\geq\frac12$ and 
$p_1\geq p_2\geq\cdots\geq p_n>0$.
Let $\AA\subset 2^{[n]}$ be a $3$-wise intersecting family. Then
\[
 \mu(\AA) \leq p_1.
\]
Moreover equality holds if and only if $\AA$ is a star centered at 
some $i\in[n]$ with $p_i=p_1$.
\end{prop}

\begin{proof}
The case $n=1$.
Let $V^{(1)}=2^{\{1\}}=\{\emptyset,\{1\}\}$, and we will define a symmetric
signed measure $\mu_3^{(1)}:V^{(1)}\times V^{(1)}\times V^{(1)}\to\R$. 
Here, for simplicity, we write $0$ and $1$ to mean $\emptyset$ and $\{1\}$, 
e.g., we write $\mu_3^{(1)}(0,1,1)$ to mean 
$\mu_3^{(1)}(\emptyset,\{1\},\{1\})$. Now $\mu_3^{(1)}$ is defined by 
\begin{align*}
& \mu_3^{(1)}(0,1,1)= \mu_3^{(1)}(1,0,1)= \mu_3^{(1)}(1,1,0)=\frac12{p_1},\quad
 \mu_3^{(1)}(0,0,0)=1-\frac32p_1,\\
& \mu_3^{(1)}(1,0,0)= \mu_3^{(1)}(0,1,0)= \mu_3^{(1)}(0,0,1)=
 \mu_3^{(1)}(1,1,1)=0.
\end{align*}
Then
\[
\mu_2^{(1)}(1,1)=\frac12p_1,\quad
\mu_2^{(1)}(1,0)=\mu_2^{(1)}(0,1)=\frac12p_1,\quad
\mu_2^{(1)}(0,0)=1-\frac32p_1,
\]
and
\[
\mu_1^{(1)}(1)=p_1,\quad\mu_1^{(1)}(0)=q_1.
\]
It follows from $0<p_1<\frac23$ that
$\mu_1^{(1)}$ and $\mu_2^{(1)}/\mu_1^{(1)}$ take values in $(0,1)$. 
So we can define a weighted 3-graph $H=(V^{(1)},\mu_3^{(1)})$.
Then, from \eqref{def:T} and \eqref{def:Tx},
we have the following matrices. 
\[
 T^{(1)}=\left[
\begin{matrix}
    1-\frac{p_1}{2q_1}&    \frac{p_1}{2q_1}\\
    \frac12 & \frac12
\end{matrix}
\right],\quad
T^{(1)}_\emptyset=\left[
\begin{matrix}
 1&0\\0&1
\end{matrix}
\right],\quad
T^{(1)}_{\{1\}}=\left[
\begin{matrix}
0&1\\ 1&0
\end{matrix}
\right].
\]
By direct computation we get the following table concerning spectral 
information.
\begin{center}
\begin{tabular}{|c||c|c|c|}
\hline
& $T^{(1)}$& $T^{(1)}_\emptyset$& $T^{(1)}_{\{1\}}$\\
\hline
eigenvalues $\lambda$ & $1,1-\frac1{2q_1}$ & $1,1$ & $1,-1$\\
\hline
$\lambda_{\min}$ & $1-\frac1{2q_1}$ & 1 & $-1$\\
\hline
\end{tabular}
\end{center}
Let $\phi$ be the indicator of a maximum independent set in $H$.
Then $\alpha(H)=\E_{\mu^{(1)}_1}[\phi]$ and $\emptyset\not\in\phi$.
So, by Theorem~\ref{thm:3-graph Hoffman}, we have
\begin{align*}
1- \alpha(H)&\geq 
\frac1{(1-\lambda_{\min}(T^{(1)}))\max_{x\in\phi}
(1-\lambda_{\min}(T^{(1)}_x))}\\
&=\frac{1}{(1-1+\frac1{2q_1})(1+1)}=q_1,
\end{align*}
and $\alpha(H)\leq 1-q_1=p_1$.

\smallskip
The general case $n\geq 2$. 
For $i=1,2,\ldots, n$ let $V_i=2^{\{i\}}$ and 
let $\mu_3^{(i)}$ be defined as in the previous $n=1$ case.
Let $H^{(i)}=(V^{(i)},\mu_3^{(i)})$ be the weighted 3-graph.
This induces the weighted 2-graph and the link graphs with
the adjacency matrices $T^{(i)},T_\emptyset^{(i)}, T_{\{i\}}^{(i)}$, where
\[
 T^{(i)}=
\left[
\begin{matrix}
    1-\frac{p_i}{2q_i}&    \frac{p_i}{2q_i}\\
    \frac12& \frac12
\end{matrix}
\right],
\quad T_\emptyset^{(i)}=T_\emptyset^{(1)},
\quad T_{\{i\}}^{(i)}=T_{\{1\}}^{(1)}.
\]
We will construct a weighted 3-graph $H=(V,\mu_3)$ from 
$H^{(1)},\ldots, H^{(n)}$.
Let $V=V^{(1)}\times \cdots\times V^{(n)}\cong 2^{[n]}$. Define 
$\mu_3:V^3\to\R$ by $\mu_3=\mu_3^{(1)}\times\cdots\times\mu_3^{(n)}$.
Let $\mu_2\in\R^{V^2}$ and $\mu_1\in\R^{V}$ be the marginals.
Then, as in Claim~\ref{mu1=mu1^1...}, we see that
$\mu_i=\mu_i^{(1)}\times\cdots\times\mu_i^{(n)}$ for $i=1,2$, as well,
in particular, $\mu_1=\mu_\pp$.
Note also that both $\mu_1$ and $\mu_2/\mu_1$ take values in $(0,1)$, 
and we need the condition $p_1<\frac23$ here.
Consequently $H$ is a weighted 3-graph with the adjacency matrix
$T=T^{(n)}\otimes\cdots\otimes T^{(1)}$.

We apply Fact~\ref{T^(i)} with $r=3$ and Fact~\ref{ONB}.
Then, for each $S\in V$, the matrix $T$ has an eigenvalue 
$\lambda_S:=\prod_{j\in S}\left(1-\frac1{2q_j}\right)$
with the corresponding eigenvector $\vv_S$ from Fact~\ref{ONB}.
Since $\frac12\leq p_1<\frac23$ and $\lambda_{\{1\}}=1-\frac1{2q_1}$
we have $-\frac12<\lambda_{\{1\}}\leq 0$ and 
$\lambda_{\min}(T)=\lambda_{\{1\}}$. 
The adjacency matrix of the link graph $H_S=(V,\mu_{2,S})$ is
$T_S=T^{(n)}_{s_n}\otimes\cdots\otimes T^{(1)}_{s_1}$, where $s_i=S\cap\{i\}$. 
If $S\neq\emptyset$ then $T_S$ has eigenvalues $\{1,-1\}$
and 
\begin{align}\label{lambda(T_S)}
\lambda_{\min}(T_S)=-1. 
\end{align}
Let $\phi$ be the indicator of a maximum independent set in $H$.
We have $\emptyset\not\in\phi$ because
\[
\mu_3(\emptyset,\emptyset,\emptyset)=\prod_{i=1}^n\mu_3^{(i)}(0,0,0)
=\prod_{i=1}^n(1-\tfrac32p_i)\neq 0.
\]
Thus, by Theorem~\ref{thm:3-graph Hoffman}, we have
\begin{align*}
1- \alpha(H)&\geq 
\frac1{(1-\lambda_{\min}(T))\max_{S\in\phi}(1-\lambda_{\min}(T_S))}=
\frac1{\left(1-(1-\frac1{2q_1})\right)(1-(-1))}=q_1,
\end{align*}
and $\alpha(H)\leq 1-q_1=p_1$. 

Now let $\AA\subset 2^{[n]}$ be a $3$-wise intersecting family. 
If $A,B,C\in\AA$ then there is some $i\in A\cap B\cap C$. 
Then $\mu_3(A,B,C)=0$ follows from the fact that 
$\mu_3^{(i)}(\{i\},\{i\},\{i\})=0$ with the definition of $\mu_3$. 
This means that $\AA\subset V$ is an independent set in $H$. 
Since $\mu_1=\mu_1^{(1)}\times\cdots\times\mu_n^{(n)}=\mu_\pp$
we have $\mu_{\pp}(\AA)=\mu_1(\AA)\leq\alpha(H)\leq p_1$,
which completes the proof of inequality.

Finally we show the uniqueness of equality case.
Suppose that $\alpha(H)=p_1$ and let $\phi$ be the indicator of
a maximum independent set $I$ in $H$. 
Then $\emptyset\not\in I$ and $I$ is also an independent set in the 
link graph $H_S=(V,\mu_{2,S})$ if $S\neq\emptyset$.
Thus by applying Theorem~\ref{thm:Hoffman bound for 2-graph} to $H_S$
with \eqref{lambda(T_S)} we have
\begin{align}\label{E_1,x<1/2}
 \max_{S\in\phi}\E_{\mu_1,S}[\phi]\leq\max_{S\neq\emptyset}
\frac{-\lambda_{\min}(T_S)}{1-\lambda_{\min}(T_S)}=\frac12.
\end{align}
(We note that \eqref{E_1,x<1/2} holds for the indicator of
\emph{any} independent set, not necessarily a maximum one, and we will use
this fact in the next subsection.)
Then by Lemma~\ref{claim:E[f,f]} we have 
$\E_{\mu_2}[\phi,\phi] \leq \frac{p_1}2$.
This together with $\hat\phi_0=p_1$ and $\lambda_{\min}(T)=1-\frac1{2q_1}$ 
verifies \eqref{eq:red1}, and we can apply Lemma~\ref{reduction1}.
Since $\lambda_{\min}(T)$ is attained only by $\lambda_{\{i\}}$
with $i\in J:=\{j\in[n]:p_j=p_1\}$ we have
$\phi=p_1\one+\sum_{j\in J}\hat\phi_{\{j\}}\vv_{\{j\}}$.
Finally by Lemma~\ref{reduction2} it follows that $\phi$ is the indicator
of a star centered at some $i\in J$.
\end{proof}

\begin{proof}[Proof of Theorem~\ref{thm2}]
We note that the $\mu_1$ in Theorem~\ref{thm1} and the $\mu_1$ in
Proposition~\ref{prop1} are the same, and moreover $\mu_\pp=\mu_1$.
Then Theorem~\ref{thm2} for the case $p_1\leq\frac12$ follows from
Theorem~\ref{thm1}, and the case $\frac12\leq p_1<\frac23$ follows from
Proposition~\ref{prop1}.
Thus we may assume that $p_1\geq\frac23$ and $p_2<\frac 23$.
Now we follow the argument in \cite{FFFLO}.
Let $\pp=(p_1,p_2,p_3,\ldots,p_n)$ and $\pp'=(p_2,p_2,p_3,\ldots,p_n)$,
that is, $\pp'$ is obtained from $\pp$ by replacing $p_1$ with $p_2$.
Let $\AA$ be the star centered at $1$, 
and let $\BB$ be an inclusion maximal intersecting family.
Suppose that $\BB\neq\AA$, and 
we will show that $\mu_\pp(\AA)>\mu_\pp(\BB)$.

By the construction we have $\mu_\pp(\AA)=p_1$ and $\mu_{\pp'}(\AA)=p_2$.
Thus $\mu_\pp(\AA)=\frac{p_1}{p_2}\mu_{\pp'}(\AA)$.

On the other hand, by Proposition~\ref{prop1}, we have
$\mu_{\pp'}(\BB)\leq p_2$.
Let $B\in\BB$. If $1\in B$ then $\mu_\pp(B)=\frac{p_1}{p_2}\mu_{\pp'}(B)$.
If $1\not\in B$ then $\mu_\pp(B)=\frac{q_1}{q_2}\mu_{\pp'}(B)<
\frac{p_1}{p_2}\mu_{\pp'}(B)$, where we used $p_1>p_2$.
Since $\BB\neq\AA$ and $\BB$ is inclusion maximal there is some $B\in\BB$
such that $1\not\in B$, e.g., $\{2,3,\ldots,n\}\in\BB$. Thus we have
$\mu_\pp(\BB)<\frac{p_1}{p_2}\mu_{\pp'}(\BB)\leq p_1=\mu_\pp(\AA)$, as needed.
This means that $\AA$ is the only intersecting family which attains the 
maximum $\mu_\pp$-measure.
\end{proof} 

\subsection{Stability: Proof of Theorem~\ref{thm3}}

Friedgut \cite{Friedgut} obtained a stability result for 2-wise 
$t$-intersecting families. The special case $t=1$,
which is a stability version of a result by Ahlswede--Katona \cite{AKa},
reads as follows. 

\begin{prop}[\cite{Friedgut}]\label{stability 2-wise}
Let $0<p<\frac 12$ be fixed. Then there exists a constant $\epsilon_p>0$ 
such that the following holds for all $0<\epsilon<\epsilon_p$.
If $\AA\subset 2^{[n]}$ is a $2$-wise intersecting family with 
$\mu_\pp(\AA)=p-\epsilon$, then there exists a star $\BB\subset 2^{[n]}$ 
such that $\mu_\pp(\AA\triangle\BB)<(C_p+o(1))\epsilon$, where
$\pp=(p,p,\ldots,p)$ and $C_p=\frac{4q^2}{1-2p}$.
\end{prop}
\noindent
We include the proof for convenience in Appendix. 
(The $C_p$ is not explicitly computed in \cite{Friedgut}.)

In this section we adapt his proof to 3-wise intersecting families to
show the following.

\begin{prop}\label{stability 3-wise}
Let $\frac12<p<\frac 23$ be fixed. Then there exists a constant $\epsilon_p>0$ 
such that the following holds for all $0<\epsilon<\epsilon_p$.
If $\AA\subset 2^{[n]}$ is a $3$-wise intersecting family with 
$\mu_\pp(\AA)=p-\epsilon$, then there exists a star $\BB\subset 2^{[n]}$ such
that $\mu_\pp(\AA\triangle\BB)<(C_p+o(1))\epsilon$, where
$\pp=(p,p,\ldots,p)$ and
\[
 C_p=\frac{16pq^2}{(2p-1)(3-4p)}.
\]
\end{prop}

For the proof we use the Kindler--Safra theorem, which extends the 
Friedgut--Kalai--Naor theorem \cite{FKN}. 
To state the result we need a definition. 
Let $V=2^{[n]}$ and $g\in\{0,1\}^V$. We say that a boolean function 
$g\in\{0,1\}^V$ 
\emph{depends on at most one coordinate} if $g$ is one of the following:
\begin{itemize}
\item[(G1)] there is some $i\in[n]$ such that 
$g=\one_{\{i\}}$, or
\item[(G2)]  there is some $i\in[n]$ such that 
$g=\one-\one_{\{i\}}$, or
\item[(G3)]  $g$ is a constant function, that is, $g=\zero$, or $g=\one$.
\end{itemize} 
(G1) means that $g$ is the indicator of the star centered at $i$,
and (G2) means that $g$ is the indicator of the complement of the star.

We can expand any boolean function $f\in\{0,1\}^V$ as 
$f=\sum_{S\in V}\hat f_S\vv_S$, 
where $\vv_s$ is defined in Subsection~\ref{subsec:unique} 
(see Fact~\ref{ONB}).
Let $f^{>1}:=\sum_{|S|>1}\hat f_S\vv_S$, and let $\|f\|$ 
denote the square root of $\la f,f\ra_{\mu_\pp}$, where $\pp=(p,\ldots,p)$.
For example, if $f=\one_{\{i\}}$ then 
$f=p\vv_\emptyset-\sqrt{pq}\vv_{\{i\}}$, and $\|f\|=p$, $\|f^{>1}\|=0$.

\begin{thm}[Kindler--Safra, Corollary 15.2 in \cite{Kindler}, see also \cite{KS}]\label{KS thm}
Let $p\in(0,1)$ be fixed and let $V=2^{[n]}$. Let $f\in\{0,1\}^V$ and
$\|f^{>1}\|^2\leq \delta\ll p$. Then there exists $g\in\{0,1\}^V$
which depends on at most one coordinate and $\|f-g\|^2<(4+o(1))\delta$.
\end{thm} 
\noindent
The $o(1)$ term is actually smaller than $c_1 \exp(-\frac{c_2}{\delta})$,
where $c_1,c_2$ are positive constants depending only on $p$, see 
Corollary 6.1 in \cite{KS} for more details.

\begin{proof}[Proof of Proposition~\ref{stability 3-wise}]
Let $V=2^{[n]}$ and let $\mu_3\in\R^{V^3}$ be the measure defined in the proof
of Proposition~\ref{prop1}. By definition of $\mu_3$ with $\frac12<p<\frac23$ 
we have that $0<\mu_2(x,y)<1$ for all $x,y\in V$, and $\mu_1=\mu_\pp$. Thus
we can define a weighted 3-graph $(V,\mu_3)$.

Let $\phi$ be the indicator of $\AA$ and we write
\[
 \phi=\hat\phi_{\emptyset}\one+\sum_{S\neq\emptyset}\hat\phi_S\vv_S.
\]
To apply Theorem~\ref{KS thm} we need to show that $\|\phi^{>1}\|$ is small.
By Fact~\ref{fact2} (ii) we have
\[
\E_{\mu_2}[\phi,\phi]=\hat\phi_{\emptyset}^2+\sum_{|S|=1}\hat\phi_S^2 \lambda_S
+\sum_{|S|>1}\hat\phi_S^2 \lambda_S, 
\]
where $\lambda_S=(1-\frac1{2q})^{|S|}$.
Since $-\frac12<1-\frac1{2q}<0$ the minimum and the second minimum eigenvalues
come from the cases $|S|=1$ and $|S|=3$, respectively. 
So let $\lambda_1:=1-\frac1{2q}$ and $\lambda_3:=(1-\frac1{2q})^3$. Then we have
\begin{align}\label{E[f,f] for 3-wise}
\E_{\mu_2}[\phi,\phi]\geq\hat\phi_{\emptyset}^2
+\lambda_1\sum_{|S|=1}\hat\phi_S^2 +\lambda_3\sum_{|S|>1}\hat\phi_S^2.
\end{align}
Define $\tau$ by $\sum_{|S|>1}\hat\phi_S^2=\tau\hat\phi_{\emptyset}$.
Then, by Fact~\ref{fact2} (i), 
$\sum_{|S|=1}\hat\phi_S^2=\hat\phi_{\emptyset}-\hat\phi_{\emptyset}^2-\tau\hat\phi_{\emptyset}$. Thus we have
\[
\E_{\mu_2}[\phi,\phi]\geq\hat\phi_{\emptyset}^2
+\lambda_1(\hat\phi_{\emptyset}-\hat\phi_{\emptyset}^2-\tau\hat\phi_{\emptyset})+\lambda_3\tau\hat\phi_{\emptyset}.
\]
On the other hand we have $\E_{\mu_1}(\phi)=\hat\phi_\emptyset$ and
$\max_{x\in\phi}\E_{\mu_1,x}[\phi]\leq\frac12$ by \eqref{E_1,x<1/2}.
Thus it follows from Lemma~\ref{claim:E[f,f]} that 
$\E_{\mu_2}[\phi,\phi]\leq\frac12\hat\phi_{\emptyset}$.
So estimating $\E_{\mu_2}[\phi,\phi]/\hat\phi_{\emptyset}$ we get
\[
\frac12\geq\hat\phi_{\emptyset}+\lambda_1(1-\hat\phi_{\emptyset}-\tau)
+\lambda_3\tau,
\]
which yields
\[
\tau\leq\frac1{\lambda_3-\lambda_1}
\left(\frac12-\hat\phi_{\emptyset}-\lambda_1(1-\hat\phi_{\emptyset})\right)
=\frac 1{\lambda_3-\lambda_1}\cdot\frac{\epsilon}{2q}
=\frac{4q^2}{(2p-1)(3-4p)}\,\epsilon,
\]
where we used $\hat\phi_{\emptyset}=p-\epsilon$ for the first equality. Since
\[
\|\phi^{>1}\|^2=\sum_{|S|>1}\hat\phi_S^2=\tau\hat\phi_{\emptyset}<\tau p
\]
we obtain
\[
 \|\phi^{>1}\|^2<
\frac{4pq^2}{(2p-1)(3-4p)}\,\epsilon.
\]
By applying Theorem~\ref{KS thm} with 
$\delta=\frac{4pq^2}{(2p-1)(3-4p)}\epsilon$ we can find $g\in\{0,1\}^V$ 
which depends at most one coordinate and $\|\phi-g\|^2<(4+o(1))\delta$.

We claim that $g$ is the indicator of a star, that is, (G1) happens.
Note that $\frac12<p<\frac23$ and 
$p\gg\epsilon+\delta$ by the choice of $\epsilon$.
So we have $\|\phi-\zero\|^2=\|\phi\|^2=p-\epsilon\gg \delta$
and $\|\phi-\one\|^2=1-(p-\epsilon)\gg\delta$. Thus (G3) cannot happen.
If $g$ is the indicator of the complement of a star, then 
$\|g\|^2=1-p$ and 
\[
\|\phi-g\|^2\geq(\|\phi\|-\|g\|)^2=(\sqrt{p-\epsilon}-\sqrt{1-p})^2>
\frac{16pq^2}{(2p-1)(3-4p)}\epsilon=4\delta
\]
by choosing $\epsilon\ll p$ small enough (we need
to choose $\epsilon$ quite small when $p$ is close to $1/2$).
This shows that (G2) cannot happen. 
So the only possibility is (G1), as needed.
\end{proof} 

\begin{proof}[Proof of Theorem~\ref{thm3}]
Let $\AA\subset 2^{[n]}$ be a 3-wise intersecting family with
$\mu_\pp(\AA)=p-\epsilon$. 
First let $\frac12<p<\frac23$. Then (ii) of the theorem follows from 
Proposition~\ref{stability 3-wise}.

Next let $p=\frac12$. It follows from the Brace--Daykin Theorem \cite{BD} 
that if $\FF\subset2^{[n]}$ is a 3-wise intersecting family which is not 
a subfamily of a star, then 
\[
|\FF|\leq|\{F\subset[n]:|F\cap[4]|\geq 3\}|, 
\]
or equivalently $\mu_\pp(\FF)\leq\frac5{16}$, 
where $\pp=(\frac12,\ldots,\frac12)$.
Thus if $p-\epsilon>\frac5{16}$, that is, $0<\epsilon<\frac3{16}$, then 
$\AA$ is a subfamily of a star, which shows (i) of the theorem in this case.

Finally let $0<p<\frac12$. We say that $\AA$ is 2-wise 2-intersecting if
$|A\cap A'|\geq 2$ for all $A,A'\in\AA$. If $\AA$ is \emph{not} 2-wise 
2-intersecting, then there exist $A,A'\in\AA$ such that $|A\cap A'|=1$,
say, $A\cap A'=\{i\}$. In this case every $A\in\AA$ must contain $i$ due to
the 3-wise intersecting condition. Thus $\AA$ is contained in a star $\BB$
centered at $i$, and we get (i) of the theorem in this case.
The only remaining case is that $\AA$ is 2-wise 2-intersecting,
and we show that this cannot happen. For $i=0,1,\ldots$, let 
$\frac i{2i+1}\leq p\leq\frac{i+1}{2i+3}$. Then it follows from the
Ahlswede--Khachatrian theorem \cite{AK} that $\mu_\pp(\AA)\leq\mu_\pp(\GG_i)$,
where
\[
 \GG_i=\{G\subset[n]:|G\cap[2i+2]|\geq i+2\}.
\]
A direct computation shows that 
$\mu_\pp(\GG_i)=\sum_{j=0}^i\binom{2i+2}jp^{2i+2-j}(1-p)^j<p$. So by 
choosing $\epsilon<\epsilon_p:=p-\mu_\pp(\GG_i)$ we see that
$\mu_\pp(\AA)=p-\epsilon>p-\epsilon_p=\mu(\GG_i)$, a contradiction.
\end{proof}

\section{Concluding remarks}
\subsection{Generalization to $r$-graphs}
Filmus et al.\ extended the Hoffman's bound to an $r$-graph in \cite{FGL}. 
We briefly explain how to extend Theorem~\ref{thm:3-graph Hoffman} to an
$r$-graph by induction on $r$. Let $V$ be a finite set with $|V|\geq 2$
and we define a weighted $r$-graph on $V$ as follows.
\begin{defn}
Let $\mu_r:V^r\to\R$ be a symmetric signed measure.
We say that $H=(V,\mu_r)$ is a weighted $r$-graph if
$\mu_{r-1}(x_1,\ldots,x_{r-1})>0$ for all $x_1,\ldots,x_{r-1}\in V$, 
where
\[
\mu_{r-1}(x_1,\ldots,x_{r-1}):=\sum_{y\in V}\mu_r(x_1,\ldots,x_{r-1},y).
\]
\end{defn}
\noindent
For $i=r-2,r-3,\ldots,1$ we define a measure $\mu_i\in\R^{V^i}$ 
inductively by
\[
\mu_i(x_1,\ldots,x_i):=\sum_{y\in V}\mu_{i+1}(x_1,\ldots,x_i,y). 
\]
Note that $\mu_i(x_1,\ldots,x_i)>0$ for all $x_1,\ldots,x_i\in V$.

Let $\phi$ be the indicator of an independent set $I$ in the weighted
$r$-graph $H=(V,\mu_r)$.
Then Lemma~\ref{E[f,f]>} and Lemma~\ref{claim:E[f,f]} work for $H$ as well. 
Here we define 
$\mu_{r-1,x}(y_2,\ldots,y_r):=\mu_r(x,y_2,\ldots,y_r)/\mu_1(x)$.
Then $\E_{\mu_{1,x}}[\phi]$ is bounded from above by $\alpha(H_x)$, where 
$H_x=(V,\mu_{r-1,x})$ is the link $(r-1)$-graph of $H$ relative to $x$. 
By induction hypothesis we can bound $\alpha(H_x)$, and we eventually 
bound $\E_{\mu_1}[\phi]$ using Lemma~\ref{E[f,f]>} and 
Lemma~\ref{claim:E[f,f]}. To state the bound, for $s=1,\ldots,r-2$,  
and $S=\{v_1,\ldots,v_s\}$, where $v_1,\ldots, v_s\in V$, let $T_S$ be the 
adjacency matrix of the link $(r-s)$-graph relative to $S$, defined by
\begin{align}\label{matrix for link}
 (T_S)_{x,y}=\frac{\mu_{s+2}(v_1,\ldots,v_s,x,y)}{\mu_{s+1}(v_1,\ldots,v_s,x)}. 
\end{align}
Let $\lambda_s:=\min_{S}\lambda_{\min}(T_S)$ and $\min_S$ is taken over 
all $s$-element (multi)subset $S$ of $I$. Also let 
$\lambda_0:=\lambda_{\min}(T)$ where $T$ is the adjacency matrix of $H$ 
defined by \eqref{def:T}. Then the Filmus--Golubev--Lifshitz bound 
is stated as follows.
\begin{align}\label{FGL bound}
 \E_{\mu_1}[\phi]\leq 1-\prod_{s=0}^{r-2}\frac1{1-\lambda_s}. 
\end{align}
With this bound it is not difficult to extend Theorem~\ref{thm2} to an 
$r$-graph.

\begin{thm}\label{thm2 for r-wise}
Let $1>p_1\geq p_2\geq\cdots\geq p_n>0$ and $\pp=(p_1,p_2,\ldots,p_n)$. 
Let $r\geq 3$.
If $\frac{r-1}r>p_2$ and $\AA\subset 2^{[n]}$ is an $r$-wise intersecting 
family, then $\mu_\pp(\AA) \leq p_1$. Moreover equality holds if and only if 
$\AA$ is a star centered at some $i\in[n]$ with $p_1=p_i$.
\end{thm}

The proof of Theorem~\ref{thm2 for r-wise} goes exactly the same as that of
Theorem~\ref{thm2}, and the main part is the proof of the following result
which corresponds to Proposition~\ref{prop1}.

\begin{prop}
Let $\frac {r-1}r>p_1\geq\frac{r-2}{r-1}$ and 
$p_1\geq p_2\geq\cdots\geq p_n>0$.
Let $\AA\subset 2^{[n]}$ be an $r$-wise intersecting family. Then
\[
 \mu(\AA) \leq p_1.
\]
Moreover equality holds if and only if $\AA$ is a star centered at 
some $i\in[n]$ with $p_i=p_1$.
\end{prop}

\begin{proof}
The matrices for the proof are different from the ones used in the proof 
of Theorem~7.1 in \cite{FGL}, because we will introduce a parameter 
$\epsilon>0$ so that all the measures $\mu_i$ take positive values 
and \eqref{matrix for link} is well defined.
Here we only record the matrices and the corresponding eigenvalues.
Otherwise the proof is the same as the one for Proposition~\ref{prop1}.

Let $n=1$ and $V=\{\emptyset, \{1\}\}$. 
We need to define a symmetric measure $\mu_r^{(1)}$.
For this we start with a symmetric function $\mu_r^{(1)}:V^r\to \R$ 
defined by
\[
\mu_r^{(1)}(0^i,1^{r-i}) =
\begin{cases}
0 & \text{if }i=0,\\
\frac {p_1}{r-1}-\delta_1& \text{if }i=1,\\
\epsilon &\text{if }2\leq i\leq r-1,\\
1-\frac{rp_1}{r-1}-\delta_2&\text{if }i=r,
\end{cases}
\]
where $\epsilon,\delta_1,\delta_2$ are small positive constants, and
$(0^i,1^{r-i})$ in the LHS means $i$ repeated $\emptyset$ and
$r-i$ repeated $\{1\}$.
Since $\mu_r^{(1)}$ is symmetric, any permutation of $(0^i,1^{r-i})$
takes the same value. We require 
$\sum_{x\in V^r}\mu_r^{(1)}(x)=1$ for $\mu_r^{(1)}$ to be a measure,
that is,  
\[
\sum_{x\in V^r}\mu_r^{(1)}(x)=
\binom r1\left(\frac{p_1}{r-1}-\delta_1\right)
+\sum_{i=2}^{r-1}\binom ri\epsilon 
+\binom rr\left(1-\frac{rp_1}{r-1}-\delta_2\right)=1.
\]
We also require that the induced measure $\mu_1^{(1)}$ 
is the $p_1$-biased one, that is, $\mu_1^{(1)}(1)=p_1$, and so
\begin{align*}
\mu_1^{(1)}(1)&=\sum_{x_2\in V}\mu_2^{(1)}(1,x_2)=\cdots=
\sum_{(x_2,\ldots,x_r)\in V^{r-1}}\mu_r^{(1)}(1,x_2,\ldots,x_r)\\
&=\binom{r-1}1\left(\frac{p_1}{r-1}-\delta_1\right)
+\sum_{i=2}^{r-1}\binom{r-1}i\epsilon=p_1.
\end{align*}
These two requirements yield that
\[
\delta_1=\frac{2^{r-1}-r}{r-1}\epsilon,\quad
\delta_2=\frac{(2^{r-1}-1)(r-2)}{r-1}\epsilon.
\]
Then $H=(V,\mu_r^{(1)})$ is a weighted $r$-graph by choosing $\epsilon>0$
sufficiently small.
For each $s=1,\ldots,r-2$, and $S\in V^s$, the link $(r-s)$-graph
$H_S=(V,\mu_{r-s,S}^{(1)})$ is induced from $H$. Let $T^{(1)}(\epsilon)$ 
and $T_S^{(1)}(\epsilon)$ be the adjacency matrices corresponding to $H$ 
and $H_S$, respectively, and let 
$T^{(1)}=\lim_{\epsilon\to 0}T^{(1)}(\epsilon)$ and
$T^{(1)}_S=\lim_{\epsilon\to 0}T_S^{(1)}(\epsilon)$.
After a somewhat tedious but direct computation one can verify that, 
for $i\geq 2$ and $j\geq 1$,
\begin{align*}
T^{(1)}&=
\left[
\begin{matrix}
    1-\frac{p_1}{(r-1)q_1}&    \frac{p_1}{(r-1)q_1}\\
    \frac1{r-1}& 1-\frac1{r-1}
\end{matrix}
\right],
&
T_0^{(1)}&=
\left[
\begin{matrix}
    1&0\\
    0&1
\end{matrix}
\right],
&
T_{1^j}^{(1)}&=
\left[
\begin{matrix}
    0&1\\
    \frac1{r-j-1}&\frac{r-j-2}{r-j-1}
\end{matrix}
\right],
\\
T_{0^i1^j}^{(1)}&=
\frac12
\left[
\begin{matrix}
    1&1\\
    1&1
\end{matrix}
\right],
&
T_{0^i}^{(1)}&=
\frac12
\left[
\begin{matrix}
    2&0\\
    1&1
\end{matrix}
\right],
&
T_{01^j}^{(1)}&=
\frac12
\left[
\begin{matrix}
    1&1\\
    0&2
\end{matrix}
\right].
\end{align*}
The above six matrices have the corresponding eigenvalues below:
\begin{align*}
&\{1-\tfrac1{(r-1)q_1}, 1\}, && \{1,1\}, && \{-\tfrac1{r-j-1}, 1\},  \\
&\{0,1\}, && \{\tfrac12,1\}, && \{\tfrac12,1\}.
\end{align*}
Thus we have $\lambda_0^{(1)}:=\lambda_{\min}(T)=1-\frac1{(r-1)q_1}<0$,
and
$\lambda_s^{(1)}:=\min_{S\in V^s}\lambda_{\min}(T_S)=-\frac1{r-s-1}$ 
for $s=1,\ldots,r-2$. 
Let $\phi$ be the indicator of an independent set in $H$. 
Then by \eqref{FGL bound} we have 
$\E_{\mu_1}[\phi]\leq 1-\prod_{s=0}^{r-2}\frac1{1-\lambda_s^{(1)}}=p_1$. 

For the general case let $n\geq 2$ and $V=2^{[n]}$.
We define the measure $\mu_r:V^r\to\R$ by
$\mu_r:=\mu_r^{(1)}\times\cdots\times\mu_r^{(n)}$. 
Then the corresponding adjacency matrices are obtained by taking tensor 
product of the ones in the $n=1$ case. 
So $T=T^{(1)}\otimes\cdots\otimes T^{(n)}$ with eigenvalues
\begin{align}\label{lambda_v}
\lambda_v(T):=\prod_{i\in v}(1-\frac1{(r-1)q_i})  
\end{align}
for $v\in V$, and 
$\lambda_0:=\min_{v\in V}\lambda_v(T)=\lambda_{\{1\}}
=1-\frac1{(r-1)q_1}$. 
For $1\leq s\leq r-2$ and $S\in V^s$ we have
$T_S=T_S^{(1)}\otimes\cdots\otimes T_S^{(n)}$ with
\begin{align}\label{lambda_s}
\lambda_s:=\min_{S\in V^s}\lambda_{\min}(T_S)=\lambda_{\min}(T_{1^s})=
-\frac1{r-s-1}. 
\end{align}
Finally it follows form \eqref{FGL bound} that $\E_{\mu_1}[\phi]\leq p_1$.
\end{proof}

\begin{conj}
The condition $\frac{r-1}r>p_2$ in Theorem~\ref{thm2 for r-wise} can be 
replaced with $\frac{r-1}r>p_{r+1}$.
In particular, Theorem~\ref{thm2} holds if $p_4<\frac 23$ instead of
$p_3<\frac 23$.
\end{conj}
\noindent
On the other hand, the condition above cannot be replaced with 
$\frac{r-1}r>p_{r+2}$. To see this let
$\AA=\{A\in 2^{[n]}:|A\cap[r+1]|\geq r\}$,
and $p_1=\cdots=p_{r+1}=:p$.
Then $\AA$ is an $r$-wise intersecting family with 
$\mu_\pp(\AA)=(r+1)p^rq+p^{r+1}$. A computation shows $\mu_\pp(\AA)$
is greater than $p$ provided, e.g., $p\geq 1-\frac1{r^2}$. 
More generally we can ask the following.
\begin{prob}
Let $1>p_1\geq p_2\geq\cdots\geq p_n>0$ and $\pp=(p_1,p_2,\ldots,p_n)$. 
Determine the maximum of $\mu_\pp(\AA)$, where $\AA\subset 2^{[n]}$ is an 
$r$-wise intersecting family.
\end{prob}

Proposition~\ref{stability 3-wise} can be extended to $r$-wise intersecting
families as follows.
\begin{prop}\label{stability r-wise}
Let $r\geq 3$ and $\frac{r-2}{r-1}<p<\frac {r-1}r$ be fixed. 
Then there exists a constant $\epsilon_{r,p}>0$ 
such that the following holds for all $0<\epsilon<\epsilon_{r,p}$.
If $\AA\subset 2^{[n]}$ is an $r$-wise intersecting family with 
$\mu_\pp(\AA)=p-\epsilon$, then there exists a star $\BB\subset 2^{[n]}$ such
that $\mu_\pp(\AA\triangle\BB)<(C_p+o(1))\epsilon$, where
$\pp=(p,p,\ldots,p)$ and
\[
 C_p=\frac{4(r-1)^2pq^2}{\left((r-1)p-(r-2)\right)\left((2r-3)-2(r-1)p\right)}.
\]
\end{prop}

\begin{proof}
The proof is the same as the proof of Proposition~\ref{stability 3-wise}.
We estimate \eqref{E[f,f] for 3-wise} from both sides. For the RHS we 
see from \eqref{lambda_v} that the minimum and the second minimum eigenvalues 
come from the cases $|v|=1$ and $|v|=3$, so
\[
 \lambda_1=1-\frac1{(r-1)q}, \quad
 \lambda_3=\lambda_1^3.
\]
For the LHS we use Lemma~\ref{claim:E[f,f]} with \eqref{lambda_s}, and we have
\[
 \E_{\mu_2}[\phi,\phi]\leq\phi_{\emptyset}
\left(1-\prod_{s=1}^{r-2}\frac1{1-\lambda_s}\right)=1-\frac1{r-1}.
\]
Then we get the $C_p$ exactly in the same way as in the proof of 
Proposition~\ref{stability 3-wise}.
\end{proof}

\subsection{More about stability}
Let $0<p<1$ and $\pp=(p,\ldots,p)\in(0,1)^n$.
In Theorem~\ref{thm3} the statement of stability differs between the two cases
(i) and (ii). 
In (ii) (the case $p>\frac12$) we have the following example:
\[
 \AA_n:=\left(\{A\in 2^{[n]}:1\in A,\,|A|>\tfrac n2\}\setminus\{[1]\}\right)
\sqcup\{[n]\setminus[1]\}.
\]
Then $\AA_n$ is a 3-wise intersecting family with $\mu_\pp(A_n)\to p$ as 
$n\to\infty$, but $\AA_n$ is not contained in any star.
It is worth noting that the stability in (i) (the case $p\leq\frac12$) also 
differs from the situation in 2-wise intersecting case in 
Proposition~\ref{stability 2-wise}. Indeed let
\[
 \AA'_n:=\left(\{A\in 2^{[n]}:1\in A\setminus\{\{1\}\}\right)\cup\{[n]\setminus\{1\}\},
\]
then $\AA'_n$ is a 2-wise intersecting family with
$\mu_\pp(\AA'_n)\to p$, but no star can contain $\AA_n$.

The constant $C_p$ in Theorem~\ref{thm3} becomes very large when $p$ is 
slightly more than $\frac12$. This is because our proof relies on 
Theorem~\ref{KS thm} and we need to distinguish our indicator from the
indicator of (G2). But the family corresponding to (G2) is not 3-wise 
intersecting at all. This suggests that the $C_p$ could be far from the 
best possible value especially when $p$ is close to $\frac12$, 
or even more bravely, we conjecture the following.
\begin{conj}
 There exists $C_p'$ such that the inequality in (ii) of Theorem~\ref{thm3} 
can be replaced with 
$\mu_{\pp}(\AA\triangle\BB)<(C_p'+o(1))\epsilon$, where $C_p'\leq C_p$ 
and moreover $C_p'$ is increasing in $p$ for $\frac12\leq p<\frac23$.
\end{conj}

The item (ii) of Theorem~\ref{thm3} can be extended to $r$-wise intersecting
case as in Proposition~\ref{stability r-wise}. So maybe the item (i) 
could be extended to $r$-wise intersecting case as well.

\begin{prob}
Let $r\geq 4$ and $p\leq\frac{r-2}{r-1}$.
Is it true that the item (i) of Theorem~\ref{thm3}
holds as well for $r$-wise intersecting families?
\end{prob}
It is also interesting to see whether or not Theorem~\ref{thm3} 
(and/or Theorem~\ref{KS thm}) can be extended to a general 
$\pp=(p_1,p_2,\ldots,p_n)$.

\begin{prob}
What happens if we replace $\pp=(p,p,\ldots,p)$ in Theorem~\ref{thm3}
with $\pp=(p_1,p_2,\ldots,p_n)$ where 
$\frac 23>p_1\geq p_2\geq\cdots\geq p_n$? 
\end{prob}

\subsection{Multiply cross intersecting families}
We say that $r$ families $\AA_1,\AA_2,\ldots,\AA_r\subset 2^{[n]}$ are
\emph{$r$-cross intersecting} if 
$A_1\cap A_2\cap\cdots\cap A_r\neq\emptyset$ for all 
$A_1\in\AA_1,A_2\in\AA_2,\ldots,A_r\in\AA_r$.
Let $\pp_1,\pp_2,\ldots,\pp_r\in(0,1)^n$ be given vectors. Then one can ask the
maximum of $\prod_{i=1}^r\mu_{\pp_i}(\AA_i)$ for $r$-cross intersecting 
families. For the case $r=2$, Suda et al.\ obtained the following result.

\begin{thm}[\cite{STT}]\label{STT-thm}
For $i=1,2$ let $\pp_i=(p_i^{(1)},\ldots,p_i^{(n)})$,
and $p_i=\max\{p_i^{(\ell)}:\ell\in [n]\}$.
Suppose that $p_1^{(\ell)},p_2^{(\ell)}\leq 1/2$ for $\ell\geq 2$.
If $\AA_1,\AA_2\subset 2^{[n]}$ are $2$-cross intersecting, then
\[
\mu_{\pp_1}(\AA_1) \mu_{\pp_2}(\AA_2)\leq p_1p_2. 
\]
Moreover, unless $p_1=p_2=1/2$ and $|w|\geq 3$, equality holds if and 
only if both $\AA_1$ and $\AA_2$ are the same star centered at some 
$\ell\in w$, where
$w:=\bigl\{\ell\in [n]:(p_1^{(\ell)},p_2^{(\ell)})=(p_1,p_2)\bigr\}$. 
\end{thm}

Almost nothing is known for the cases $r\geq 3$.
Perhaps the easiest open problem is the case when $r=3$ and 
$\pp_i=(p,p,\ldots,p)$ for all $1\leq i\leq 3$.

\begin{conj}
Let $p\leq\frac23$ and $\pp=(p,p,\ldots,p)$. 
If $\AA_1,\AA_2,\AA_3\subset 2^{[n]}$ are $3$-cross intersecting, then 
$\mu_{\pp}(\AA_1)\mu_{\pp}(\AA_2)\mu_{\pp}(\AA_3)\leq p^3$.
\end{conj}

\section{Acknowledgment}
The author thanks Tsuyoshi Miezaki for valuable discussions. He also thanks
the referees for their very careful reading and helpful suggestions. 
This research was supported by JSPS KAKENHI Grant No. 18K03399.

\section{Appendix}
\begin{proof}[Proof of Fact~\ref{fact1}]
(i):
\begin{align*}
 \la f,Tg\ra_{\mu_1}&=\sum_xf(x)(Tg)(x)\mu_1(x) \\
&=\sum_xf(x)\left(\sum_y\frac{\mu_2(x,y)}{\mu_1(x)}g(y)\right)\mu_1(x) \\
&=\sum_x\sum_yf(x)g(y)\mu_2(x,y)\\
&=\E_{\mu_2}[f,g].
\end{align*}
(ii):
\[
\la f,Tg\ra_{\mu_1}=
 \E_{\mu_2}[f,g]= \E_{\mu_2}[g,f]=\la g,Tf\ra_{\mu_1}=\la Tf,g\ra_{\mu_1}.
\]
(iii):
\[
 (T\one)(x)=\sum_y\frac{\mu_2(x,y)}{\mu_1(x)}\one(y)=
\frac1{\mu_1(x)}\sum_y{\mu_2(x,y)}=1=\one(x).
\]
(iv):
\[
 \la\one,\one\ra_{\mu_1}=\sum_x\one(x)\one(x)\mu_1(x)=\sum_x\mu_1(x)=1.
\]
(v):
\[
 \la \phi,\one\ra_{\mu_1}=\sum_x \phi(x)\one(x)\mu_1(x)
=\sum_x \phi(x)\mu_1(x)=\E_{\mu_1}[\phi].
\]
(vi):
\[
 \la \phi, \phi\ra_{\mu_1}=\sum_{x} \phi(x)^2\mu_1(x)
=\sum_{x}\phi(x)\mu_1(x)=\E_{\mu_1}[\phi].
\]
\end{proof}

\begin{proof}[Proof of Fact~\ref{fact2}]
(i):
We have $\E_{\mu_1}[\phi]=\la \phi,\one\ra_{\mu_1}=\hat\phi_0$ and 
\begin{align*}
\E_{\mu_1}[\phi]&=\la \phi,\phi\ra_{\mu_1}=\la \hat\phi_0\one+
\sum_{i\geq 1}\hat\phi_i\vv_i,\hat\phi_0\one+\sum_{i\geq 1}\hat\phi_i\vv_i\ra_{\mu_1}
=\hat\phi_0^2+\sum_{i\geq 1} \hat\phi_i^2.
\end{align*}
Thus $\sum_{i\geq 1} \hat\phi_i^2=\E_{\mu_1}[\phi]-\hat\phi_0^2=\E_{\mu_1}[\phi]-\E_{\mu_1}[\phi]^2$,
which we will use to show (ii).

(ii):
\begin{align*}
\E_{\mu_2}[\phi,\phi] &=
\la \phi,T \phi\ra_{\mu_1}\\
&= \la \hat\phi_0\one+\sum_{i\geq 1} \hat\phi_i\vv_i,T( \hat\phi_0\one+\sum_{i\geq 1} \hat\phi_i\vv_i)\ra_{\mu_1}\\
&= \la \hat\phi_0\one+\sum_{i\geq 1} \hat\phi_i\vv_i,\hat\phi_0\one+\sum_{i\geq 1} \hat\phi_il_i\vv_i\ra_{\mu_1}\\
&= \hat\phi_0^2+\sum_{i\geq 1} \hat\phi_i^2l_i.
\end{align*}

\end{proof}

\begin{proof}[Proof of Proposition~\ref{stability 2-wise}]
The proof is almost identical to the proof of 
Proposition~\ref{stability 3-wise}, and we only give a sketch.
The only difference is that in this case we have $\E_{\mu_2}[\phi,\phi]=0$,
which makes things easier. Then by \eqref{E[f,f] for 3-wise} we have
\[
0\geq\hat\phi_{\emptyset}+\lambda_1(1-\hat\phi_{\emptyset}-\tau)
+\lambda_3\tau,
\]
where $\hat\phi_{\emptyset}=p-\epsilon$,
$\lambda_1=-\frac{p_1}{q_1}$, and $\lambda_3=(-\frac{p_1}{q_1})^3$.
Rearranging we have
\[
 \tau\leq\frac{q^2\epsilon}{p(1-2p)},
\]
and $\|\phi^{>1}\|^2<\frac{q^2\epsilon}{1-2p}=:\delta$.
Then using Theorem~\ref{KS thm} with this $\delta$ we get the result.
\end{proof}

\end{document}